\documentclass[12pt]{amsart}
\usepackage{amsfonts, amssymb, amsmath, amsthm}

\usepackage{color}
\usepackage{enumerate}
\usepackage{graphicx}
\usepackage{tikz-cd}

\mathchardef\ordinarycolon\mathcode`\:
\mathcode`\:=\string"8000
\begingroup \catcode`\:=\active
  \gdef:{\mathrel{\mathop\ordinarycolon}}
\endgroup

\newcommand{\xydr}{\ar@<-.5ex>[r] \ar@<.5ex>[r]}
\newcommand{\xydd}{\ar@<-.5ex>[d] \ar@<.5ex>[d]}

\newcommand{\jtower}[4]{
\xymatrix@R=1em{
#1_{1}  \ar[d]^{#2_{1}} & #1_{2} \ar[d]_{#2_{2}} \\
#1_{1} \ar[d] & #2_{2} \ar[d] \\
#3_{1} \ar[r]^{#4} & #3_{2} \\
}
}

\DeclareMathOperator{\Aut}{Aut}

\numberwithin{equation}{section}
\newtheorem{theorem}[equation]{Theorem}
\newtheorem{proposition}[equation]{Proposition}

\newtheorem{question}[equation]{Question}

\newtheorem{corollary}[equation]{Corollary}
\newtheorem{lemma}[equation]{Lemma}

\theoremstyle{definition}

\newtheorem{definition}[equation]{Definition}

\theoremstyle{remark}
\newtheorem{remark}[equation]{Remark}

                            {\end{enumerate}}
\input xy
\xyoption{all}
\numberwithin{equation}{section}
\begin{document}

\thanks{This research was supported in part by the National Science Foundation grant ``RTG: Analysis on manifolds'' at Northwestern University}

\title[Lyapunov exponents and entropy for automorphisms of the shift]{Automorphisms of the shift: Lyapunov Exponents, entropy, and the dimension representation}
\author{Scott Schmieding}
\address{Northwestern University}
\email{schmiedi@math.northwestern.edu}

\begin{abstract} Let $(X_{A},\sigma_{A})$ be a shift of finite type and $\Aut(\sigma_{A})$ its corresponding automorphism group. Associated to $\phi \in \Aut(\sigma_{A})$ are certain Lyapunov exponents $\alpha^{-}(\phi), \alpha^{+}(\phi)$ which describe asymptotic behavior of the sequence of coding ranges of $\phi^{n}$. We give lower bounds on $\alpha^{-}(\phi), \alpha^{+}(\phi)$ in terms of the spectral radius of the corresponding action of $\phi$ on the dimension group associated to $(X_{A},\sigma_{A})$. We also give lower bounds on the topological entropy $h_{top}(\phi)$ in terms of a distinguished part of the spectrum of the action of $\phi$ on the dimension group, but show that in general $h_{top}(\phi)$ is not bounded below by the logarithm of the spectral radius of the action of $\phi$ on the dimension group.
\end{abstract}
\maketitle

\tableofcontents

\section{Introduction}
By a subshift $(X,\sigma)$ we mean a closed shift-invariant subset $X \subset \mathcal{A}^{\mathbb{Z}}$, where $\mathcal{A}$ is a finite alphabet, together with the shift map $\sigma \colon X \to X$. An automorphism of $(X,\sigma)$ is a homeomorphism $\phi \colon X \to X$ such that $\phi \sigma = \sigma \phi$, and we let $\Aut (\sigma)$ denote the group of automorphisms of $(X,\sigma)$. By the Curtis-Hedlund-Lyndon Theorem (see \cite[Sec. 1.5]{LindMarcus1995}), any automorphism $\phi \in \Aut(\sigma)$ is given by a block code of finite range, and $\Aut(\sigma)$ is thus countable. When $(X,\sigma)$ is a mixing shift of finite type, $\Aut(\sigma)$ is known to contain a large assortment of subgroups: in this case $\Aut(\sigma)$ contains, for example, (isomorphic copies of) the direct sum of countable many copies of $\mathbb{Z}$, the free group on two generators, and any finite group \cite{BLR88}. In contrast, recent work (e.g. \cite{CyrKraBeyondTransitivity}, \cite{CyrKraStretchedExponential}, \cite{CyrKraSubquadratic}, \cite{DDMP2016}) shows that, in cases where the shift is of low complexity, the structure of the automorphism groups can be much more restricted. \\
\indent The analysis of certain types of distortion occurring in $\Aut(\sigma)$ was undertaken in \cite{CFKSpacetime} and \cite{CFKPDistortion}.
With an eye toward studying a given individual element $\phi \in \Aut(\sigma)$, one may consider the sequence of sizes of smallest possible coding ranges for $\phi^{n}$. When this sequence grows sublinearly, the automorphism $\phi$ is called \emph{range distorted}, a term introduced in \cite[Def. 5.8]{CFKSpacetime}. More generally, asymptotic information about this sequence is captured by certain Lyapunov exponents $\alpha^{-}(\phi),\alpha^{+}(\phi)$ (defined in \ref{def:lyaps} below), which were also studied in \cite{CFKSpacetime} and \cite{NasuDegreesResolving}. The automorphism $\phi$ is range distorted precisely when $\alpha^{-}(\phi) = \alpha^{+}(\phi) = 0$.
(We note that \cite{CFKPDistortion} also considers an alternative notion of distortion, which is more classical and group theoretic in nature; the relationship between the two is discussed in Section 5 below.) The quantities $\alpha^{-}, \alpha^{+}$ have been previously studied in \cite{Shereshevsky1992} and \cite{Tisseur2000} from the point of view of measure-preserving cellular automata. Similar quantities have also been examined in \cite{PacificoVieitez}, in the more general context of expansive homeomorphisms of a compact metric space.\\
\indent Automorphisms of finite order are automatically range distorted. However, there exist range distorted automorphisms of infinite order. Recently, Guillon and Salo in \cite{GS2017} have given constructions to produce a vast collection of such infinite order range distorted automorphisms of transitive subshifts, based on the concept of aperiodic one-head machines. \\
\indent We consider here the case where $(X,\sigma)$ is an irreducible shift of finite type. For such systems, there is a homomorphism $\pi \colon \Aut(\sigma) \to \Aut(\mathcal{G},\mathcal{G}^{+},\delta)$ to automorphisms of a certain dimension triple (defined in Section \ref{sec:dimensiongroupdef}). An automorphism $\pi(\phi)$ extends in a natural way to an automorphism of a finite dimensional vector space, allowing one to study an automorphism $\phi$ through its associated linear map. In this paper we use this approach to study automorphisms of shifts of finite type, and examine connections between their dimension representation, Lyapunov exponents, entropy, and distortion.\\
\indent An outline of the paper is as follows. Section \ref{sec:dimensiongroupdef} provides relevant background on the dimension triple and the dimension representation associated to a shift of finite type.\\
\indent Section \ref{sec:entropythings} considers relations between entropy and the dimension representation. The section's main result (Theorem \ref{thm:perronentropybound} in the text) shows that the topological entropy of $\phi$ is bounded below by $\log \lambda_{\phi}$, where $\lambda_{\phi}$ is a certain distinguished eigenvalue of the linear map $\pi(\phi)$. This eigenvalue $\lambda_{\phi}$ can be interpreted as the value by which $\phi$ scales a canonical (up to a scalar) family of $\sigma$-finite measures on unstable sets of the system $(X_{A},\sigma_{A})$. Analogous to Shub's classical entropy conjecture \cite{Shub1974} (additionally, see \cite{KatokECRussian}), a natural question is whether $h_{top}(\phi)$ is always bounded below by the logarithm of the spectral radius of $\pi(\phi)$. We show that, in general, this is false, and construct examples where such a bound does not hold. \\
\indent Section \ref{sec:lyapunovbois} discusses Lyapunov exponents and their connection with the dimension representation. The section's main result (Theorem \ref{thm:main} in the text) gives lower bounds on the Lyapunov exponents $\alpha^{-}(\phi), \alpha^{+}(\phi)$ in terms of the spectral radius of the induced linear map $\pi(\phi)$. As a consequence, we prove that if both $\phi$ and $\phi^{-1}$ are range distorted, then the spectrum of the linear map $\pi(\phi)$ must lie on the unit circle.\\
\indent In Section 5 we briefly discuss how group distorted elements behave with respect to the dimension representation.\\
\indent We also highlight three open questions (Questions \ref{question:maxentropybound}, \ref{question:zepowerinert}, and \ref{question:rdpowerinert} in the text).\\
\indent I would like to thank Ville Salo for communicating to us the example in Remark \ref{remark:connectionremark}, and thank Masakazu Nasu for referring me to results contained in his article \cite{NasuDegreesResolving}. I am especially thankful to Mike Boyle for many helpful discussions and comments.\\

For a square matrix $A$ over $\mathbb{Z}_{+}$ we let $(X_{A},\sigma_{A})$ denote the edge shift of finite type associated to $A$. \emph{Throughout, we will assume that $A$ is irreducible, so $(X_{A},\sigma_{A})$ is an irreducible shift of finite type. We will also assume that $(X_{A},\sigma_{A})$ has positive entropy.}


\section{The Dimension Representation}\label{sec:dimensiongroupdef}
Let $A$ be a $k \times k$ square matrix over $\mathbb{Z}_{+}$. We let $R(A)$ denote the eventual range subspace $\mathbb{Q}^{k}A^{k} \subset \mathbb{Q}^{k}$ (throughout we will always assume matrices to act on row vectors). The \emph{dimension triple $(\mathcal{G}_{A},\mathcal{G}_{A}^{+},\delta_{A})$ associated to $A$} consists of the abelian group $\mathcal{G}_{A}$, the semi-group $\mathcal{G}_{A}^{+} \subset \mathcal{G}_{A}$, and the automorphism $\delta_{A}$ of $\mathcal{G}_{A}$, where
\begin{enumerate}
\item
$\mathcal{G}_{A} = \{ x \in R(A) \mid x A^{k} \in \mathbb{Z}^{k} \textnormal{ for some }k \ge 0\}$.
\item
$\mathcal{G}_{A}^{+} = \{x \in R(A) \mid x A^{k} \in (\mathbb{Z}_{+})^{k} \textnormal{ for some } k\ge 0\}$.
\item
$\delta_{A}(x) = x A$.
\end{enumerate}
\indent While the definition of $(\mathcal{G}_{A},\mathcal{G}_{A}^{+},\delta_{A})$ above relies on the matrix $A$, there is an alternative definition, due to Krieger, which is built more directly from the system $(X_{A},\sigma_{A})$. We will make use of Krieger's presentation, which we now outline. Our presentation and terminology follows that of \cite[Sec. 7.5]{LindMarcus1995}, and for more details, we refer the reader there. Recall we are assuming that $A$ is a $k \times k$ irreducible matrix.\\
\indent By an \emph{$m$-ray} we mean a subset of $X_{A}$ given by
$$R(x,m) = \{y \in X_{A} \mid y_{(-\infty,m]}=x_{(-\infty,m]}\}$$
for some $x \in X_{A}, m \in \mathbb{Z}$. An \emph{$m$-beam} is a finite union of $m$-rays. By a \emph{ray} (\emph{beam}) we mean an $m$-ray ($m$-beam) for some $m \in \mathbb{Z}$. We note that if $U$ is an $m$-beam for some $m$, and $n \ge m$, then $U$ is also an $n$-beam. Given an $m$-beam
$$U = \bigcup_{i=1}^{j}R(x^{(i)},m),$$
we let $v_{U,m} \in \mathbb{Z}^{k}$ denote the vector whose $J$th component is given by
$$\#\{x^{(i)} \in U \mid \textnormal{ the edge corresponding to }x_{m}^{(i)} \textnormal{ ends at state } J\}.$$
We define beams $U$ and $V$ to be equivalent if there exists $m$ such that $v_{U,m} = v_{V,m}$, and let $[U]$ denote the equivalence class of a beam $U$. Since $A$ is irreducible and $0 < h_{top}(\sigma_{A}) = \log \lambda_{A}$, given beams $U,V$, one may always find beams $U^{\prime}, V^{\prime}$ such that
$$[U]=[U^{\prime}], \hspace{.23in} [V] = [V^{\prime}], \hspace{.23in} U^{\prime} \cap V^{\prime} = \emptyset,$$
and we let $D_{A}^{+}$ denote the abelian semi-group defined by the operation
$$[U] + [V] = [U^{\prime} \cup V^{\prime}].$$
Letting $D_{A}$ denote the group completion of $D_{A}^{+}$ (so elements of $D_{A}$ are formal differences $[U]-[V]$), the map $d_{A} \colon D_{A} \to D_{A}$ induced by
$$d_{A}([U]) = [\sigma_{A}(U)]$$
is a group automorphism of $D_{A}$, and we arrive at Krieger's dimension triple $(D_{A},D_{A}^{+},d_{A})$.

An automorphism $\phi \in \Aut (\sigma_{A})$ induces an automorphism $\phi_{*} \colon (D_{A},D_{A}^{+},d_{A}) \to (D_{A},D_{A}^{+},d_{A})$ by
$$\phi_{*}([U]) = [\phi(U)].$$
Here and in what follows, by a morphism of a triple we mean a morphism preserving all the relevant data. For example, by an automorphism $\Phi \in \Aut(\mathcal{G}_{A},\mathcal{G}_{A}^{+},\delta_{A})$ we mean a group automorphism $\Phi \colon \mathcal{G}_{A} \to \mathcal{G}_{A}$ taking $\mathcal{G}_{A}^{+}$ onto $\mathcal{G}_{A}^{+}$ such that $\Phi \delta_{A} = \delta_{A} \Phi$.\\
\indent Finally, there is a semi-group homomorphism $\theta \colon D_{A}^{+} \to \mathcal{G}_{A}^{+}$ induced by the map
$$\theta([U]) = \delta_{A}^{-k-n}(v_{U,n}A^{k}), \hspace{.29in} U \textnormal{ an } n\textnormal{-beam}.$$
\begin{proposition}[\cite{LindMarcus1995}, Theorem 7.5.3]\label{prop:fernus}
The map $\theta \colon D_{A}^{+} \to \mathcal{G}_{A}^{+}$ satisfies $\theta(D_{A}^{+}) = \mathcal{G}_{A}^{+}$, and induces an isomorphism $\theta \colon D_{A} \to \mathcal{G}_{A}$ such that $\theta \circ d_{A} = \delta_{A} \circ \theta$. Thus $\theta$ induces an isomorphism of triples
$$\theta \colon (D_{A},D_{A}^{+},d_{A}) \to (\mathcal{G}_{A},\mathcal{G}_{A}^{+},\delta_{A}).$$
\end{proposition}
For $\phi \in \Aut(\sigma_{A})$ we let $S_{\phi} \colon (\mathcal{G}_{A},\mathcal{G}_{A}^{+},\delta_{A}) \to (\mathcal{G}_{A},\mathcal{G}_{A}^{+},\delta_{A})$ denote the automorphism of the dimension triple for which the diagram
\begin{equation}
\begin{aligned}\label{diagram}
\xymatrix{
D_{A} \ar^{\theta}[r] \ar_{\phi_{*}}[d] & \mathcal{G}_{A} \ar^{S_{\phi}}[d]\\
D_{A} \ar^{\theta}[r] & \mathcal{G}_{A} \\
}
\end{aligned}
\end{equation}
commutes. We can now define the dimension representation by
\begin{gather}\label{def:dimrep}
\pi_{A} \colon \Aut(\sigma_{A}) \to \Aut(\mathcal{G}_{A},\mathcal{G}_{A}^{+},\delta_{A})\\
\pi_{A} \colon \phi \mapsto S_{\phi}.
\end{gather}
An automorphism $\phi \in \Aut(\sigma_{A})$ is called \emph{inert} if it is in the kernel of the dimension representation of $\Aut(\sigma_{A})$.\\

From the linear algebra point of view, there is a rather concrete interpretation of the map $\pi_{A}$. Considering $\mathcal{G}_{A}$ as a subgroup of the rational vector space $R(A)$, the automorphism $\pi_{A}(\phi) = S_{\phi}$ extends uniquely to a linear automorphism $S_{\phi,\mathbb{Q}} \colon R(A) \to R(A)$ which preserves $\mathcal{G}_{A}$. The map $S_{\sigma_{A},\mathbb{Q}} \colon R(A) \to R(A)$ is given by $x \mapsto xA$, and the linear maps $S_{\phi,\mathbb{Q}}, S_{\sigma_{A},\mathbb{Q}}$ commute.
\section{Entropy and the dimension representation}\label{sec:entropythings}
In this section we discuss some aspects of the relationship between the entropy of an automorphism $\phi$ of an SFT $(X_{A},\sigma_{A})$ and its action on the associated dimension group $(\mathcal{G}_{A},\mathcal{G}_{A}^{+})$. The two main results are Theorem \ref{thm:perronentropybound}, and a construction of examples showing that a certain entropy conjecture does not hold in general. In short, Theorem \ref{thm:perronentropybound} is a positive result, giving lower bounds on the entropy of an automorphism $\phi$ in terms of a particular component of the spectrum (denoted by $\lambda_{\phi}$) of its action on the dimension group. This distinguished eigenvalue $\lambda_{\phi}$ has dynamical interpretations, which we discuss. Following this is a construction which gives examples showing this can not be strengthened to bound the entropy below by the logarithm of the entire spectral radius of the action on the dimension group; in particular, a stronger form of Theorem \ref{thm:perronentropybound}, which is analogous Shub's classical entropy conjecture, does not hold in general. \\
\indent Throughout, we will use the same notation as above: $(X_{A},\sigma_{A})$ denotes an irreducible shift of finite type (which we will assume has positive entropy), we let $\pi_{A} \colon \Aut(\sigma_{A}) \to \Aut(\mathcal{G}_{A},\mathcal{G}_{A}^{+})$ denote the associated dimension representation, and $S_{\phi} = \pi_{A}(\phi)$ the image of $\phi$, so that $S_{\phi} \otimes 1 \colon \mathcal{G}_{A} \otimes \mathbb{C} \to \mathcal{G}_{A} \otimes \mathbb{C}$. \\

We first outline the relevant background regarding measures on unstable sets and the quantity $\lambda_{\phi}$ which appears in Theorem \ref{thm:perronentropybound} below. The content here follows closely that of \cite[Sec. 3]{BoyleDegree}, and we refer the reader there for details and proofs.\\

\indent For a point $x \in X_{A}$ we define the unstable set of $x$ to be
$$W^{u}(x) = \bigcup_{n \in \mathbb{Z}}R(x,n).$$
We equip $W^{u}(x)$ with a topology by using the collection
$$\{R(y,m) \mid y \in W^{u}(x), m \in \mathbb{Z}\}$$
as a basis, and with this topology the space $W^{u}(x)$ becomes $\sigma$-compact. \\
\indent Let $\lambda_{A}$ denote the Perron-Frobenius eigenvalue for $A$, and choose a right eigenvector $v_{r}$ for $\lambda_{A}$. We define on each $W^{u}(x)$ a $\sigma$-finite Borel measure $\mu_{u}^{x}$ by
$$\mu_{u}^{x}(R(y,m)) = \lambda_{A}^{-m}v_{r}(t(y_{m}))$$
where $t(y_{m})$ denotes the state at which the edge $y_{m}$ ends.

\indent This collection of measures $\{\mu_{u}^{x}\}$ satisfies the following: for any $x,y \in X_{A}$ and $m \in \mathbb{Z}$
\begin{enumerate}
\item
$\mu_{u}^{x}(R(x,m)) = \lambda_{A}^{-1}\mu_{u}^{\sigma_{A}(x)}(R(\sigma_{A}(x),m-1))$.
\item
There exists $N \in \mathbb{N}$ such that if $x_{[0,N]} = y_{[0,N]}$ then $\mu_{u}^{x}(R(x,N)) = \mu_{u}^{x}(R(y,N))$.
\end{enumerate}
While the collection $\{\mu_{u}^{x}\}$ is not unique, any other such collection $\{\nu^{x}_{u}\}$ satisfies ${\nu^{x}_{u} = K \mu^{x}_{u}}$ for some universal scalar $K$ (see \cite[Prop. 3.2]{BoyleDegree}).\\

Let us fix once and for all such a collection $\mu_{u} = \{\mu_{u}^{x}\}$. There is then a corresponding state on the dimension group, i.e. a group homomorphism
$$\tau_{\mu_{u}} \colon D_{A} \to \mathbb{R}, \hspace{.23in} \tau_{\mu_{u}}(D_{A}^{+}) \subset \mathbb{R}_{+}$$
induced by defining
\begin{equation}\label{eqn:tracedefinition}
\tau_{\mu_{u}} \colon D_{A}^{+} \to \mathbb{R}, \hspace{.23in} \tau_{\mu_{u}}(R(x,n)) = \mu_{u}^{x}(R(x,n)).
\end{equation}

Since $\{\mu_{u}^{x}\}$ is unique up to a universal scalar multiple, $\tau_{\mu_{u}}$ satisfies the following property: for any $\phi \in \Aut(\sigma_{A})$, there exists $\lambda_{\phi} > 0$ satisfying
$$\tau(\phi(v)) = \lambda_{\phi} \tau(v), \hspace{.07in} \textnormal{ for any } v \in \mathcal{G}_{A}.$$
We can then define a homomorphism
\begin{equation}\label{eqn:autotracemap}
\Psi \colon \Aut(\sigma_{A}) \to \mathbb{R}_{+},  \hspace{.17in} \Psi \colon \phi \mapsto \lambda_{\phi}
\end{equation}
where $\mathbb{R}_{+}$ is the set of positive reals considered as a group under multiplication. When $\textnormal{det}(I-tA)$ is irreducible, the map $\Psi$ is injective (one proof of this can be found in Cor. 5.11 of \cite{BMT1987}. We refer the reader to \cite{BoyleKrieger1987} for more details on the map $\Psi$.)

\indent The quantity $\lambda_{\phi}$ has various interpretations, and two fundamental perspectives on $\lambda_{\phi}$ are worth briefly discussing. (We note that these two views of $\lambda_{\phi}$ follow somewhat analogously the two ways we defined the dimension group itself in Section \ref{sec:dimensiongroupdef}, i.e. internally, via Krieger's definition, or through the matrix approach.) From the internal point of view, the quantity $\lambda_{\phi}$ is determined by how the automorphism $\phi \in \Aut(\sigma_{A})$ multiplies a choice of coherent measures on unstable sets. On the other hand, from the matrix point of view, $\lambda_{\phi}$ has a rather explicit interpretation: letting $\lambda_{A}$ denote the Perron-Frobenius eigenvalue for $A$, for an eigenvector $v \in \mathcal{G}_{A} \otimes \mathbb{C}$ such that $(\delta_{A} \otimes 1)v = \lambda_{A} v$, we have $(S_{\phi} \otimes 1)v = \lambda_{\phi}v$. 

\subsection{Lower bound on entropy}\label{sec:entropybound}

\indent For an automorphism $\phi \in \Aut(\sigma_{A})$, the quantity $\lambda_{\phi}$ is an eigenvalue of $S_{\phi}$. The following shows that the logarithm of this distinguished eigenvalue always bounds the topological entropy of $\phi$ from below.
\begin{theorem}\label{thm:perronentropybound}
Let $(X_{A},\sigma_{A})$ be an irreducible shift of finite type, and let $\phi \in \Aut(\sigma_{A})$. Then
\begin{equation}\label{eqn:entropybound}
\log\lambda_{\phi} \le h_{top}(\phi).
\end{equation}
\end{theorem}
\begin{remark}\label{remark:absperron}
It follows that $\lvert \log \lambda_{\phi} \rvert \le h_{top}(\phi)$, since $h_{top}(\phi) = h_{top}(\phi^{-1})$.
\end{remark}

The inequality \eqref{eqn:entropybound} can be sharp, i.e. in the case where $\phi = \sigma_{A}$. On the other hand, there are cases where \eqref{eqn:entropybound} becomes strict. For example, for any shift of finite type $\sigma_{A}$, consider the automorphism $\sigma_{A} \times \sigma_{A}^{-1}$ of the product system $(X_{A} \times X_{A},\sigma_{A} \times \sigma_{A})$ (note this product system $(X_{A} \times X_{A},\sigma_{A} \times \sigma_{A})$ is topologically conjugate to a shift of finite type). In this case, the left hand side of \eqref{eqn:entropybound} is zero, while the right hand side is $2h_{top}(\sigma_{A})$.\\

We will prove Theorem \ref{thm:perronentropybound} at the end of the section.

\begin{remark} Given $\phi \in \Aut(\sigma_{A})$, by \cite[Theorem 2.17]{BoyleKrieger1987} there exists $k_{0} \in \mathbb{N}$ such that if $k \ge k_{0}$, then $h_{top}(\sigma^{k}\phi) = \log \lambda_{\phi} + k \log \lambda_{A} = \log \lambda_{\sigma^{k} \phi}$. Thus for any automorphism, after composing with a sufficiently high power of the shift, \eqref{eqn:entropybound} becomes equality.
\end{remark}

An irreducible shift of finite type $(X_{A},\sigma_{A})$ has a unique measure of maximal entropy $\mu_{\sigma_{A}}$, 
and any automorphism $\phi \in \Aut(\sigma_{A})$ preserves $\mu_{\sigma_{A}}$ (see \cite{CovenPaul75}).

\begin{question}\label{question:maxentropybound} If $(X_{A},\sigma_{A})$ is an irreducible shift of finite type with measure of maximal entropy $\mu_{\sigma_{A}}$ and $\phi \in \Aut(\sigma_{A})$, does the inequality
\begin{equation}\label{eqn:maxentropybound}
\log (\lambda_{\phi}) \le h_{\mu_{\sigma_{A}}}(\phi)
\end{equation}
hold?
\end{question}
\indent We note that the measures $\mu_{u}^{x}$ are determined (up to some scalar multiple) by the measure of maximal entropy $\mu_{\sigma_{A}}$ (see Remark 3.3 in \cite{BoyleDegree}).\\
\indent A positive answer to Question \ref{question:maxentropybound} would imply Theorem \ref{thm:perronentropybound} by the variational principle. However, Question \ref{question:maxentropybound} asks for something strictly stronger than Theorem \ref{thm:perronentropybound}; indeed, there are automorphisms of shifts of finite type for which the measure of maximal entropy for the automorphism does not coincide with the measure of maximal entropy for the shift. Here is an easy example of such an automorphism (we thank Mike Boyle for communicating this example to us). \\
\indent Let $\mathcal{A} = \{(a,b) \mid a,b \in \{0,1\}\}$, and consider the full shift $(X_{5},\sigma_{5})$ on the alphabet $\mathcal{A} \cup \{c\}$. Define an automorphism $\phi \in \Aut(\sigma_{5})$ by a range one block code
\begin{center}
\begin{tikzcd}[row sep=scriptsize]
{*}c{*} \arrow[mapsto]{d} &  c(a,b)(a^{\prime},b^{\prime})\arrow[mapsto]{d} & (a,b)(a^{\prime},b^{\prime})c \arrow[mapsto]{d} & (a,b)(a^{\prime},b^{\prime})(a^{\prime \prime},b^{\prime \prime}) \arrow[mapsto]{d} \\
c  & (a^{\prime},a) & (b^{\prime},b) & (a^{\prime \prime},b) 
\end{tikzcd}
\end{center}
where each ${*}$ can be any symbol. Let $\mu_{5}$ denote the measure of maximal entropy for $\sigma_{5}$, choose $n \in \mathbb{N}$, and let $P_{n}$ denote the partition of $X_{5}$ into cylinder sets given by words of length $2n+1$ centered at 0. For a $\mu_{5}$-generic point $z$ there exists $i,j > n$ such that $z_{-i}=c, z_{j} = c$, so the $\phi$-itinerary of $z$ through $P_{n}$ is eventually periodic. It follows that $h_{\mu_{5}}(\phi) = 0$. However, if $Y \subset X_{5}$ denotes the set of points which never contain a symbol $c$, then $\phi |_{Y}$ is conjugate to the product of the full 2-shift with its inverse, so $h_{top}(\phi) \ge \log(4)$.\\
\indent Examples where an automorphism and its shift have different measures of maximal entropy need not rely on the appearance of equicontinuity: there are automorphisms which are conjugate to a shift of finite type, but the automorphism and the shift map do not have the same measure of maximal entropy. An explicit example of such an automorphism can be found in \cite[Sec. 10]{Nasu1995}. \\

In some cases, zero entropy implies inertness, as the following shows.
\begin{corollary}\label{cor:entropyzero}
Suppose $(X_{A},\sigma_{A})$ is an irreducible shift of finite type such that the polynomial $\textnormal{det}(I-tA)$ is irreducible. If $\phi \in \Aut(\sigma_{A})$ satisfies $h_{top}(\phi) = 0$, then $\phi$ is inert.
\end{corollary}
\begin{proof}
If $h_{top}(\phi) = 0$, then $\lambda_{\phi} = 1$ by Theorem \ref{thm:perronentropybound}. Since $\textnormal{det}(I-tA)$ is irreducible, the map $\Psi$ from \eqref{eqn:autotracemap} is injective, and it follows that $\phi$ is inert.
\end{proof}

There exist mixing shifts of finite type which have non-inert automorphisms of finite order (and hence zero entropy). Such automorphisms are not hard to produce, but here is an explicit example, for completeness. Let $B = \begin{pmatrix} 2 & 1 \\ 1 & 2 \end{pmatrix}$, and let $\phi \in \Aut(\sigma_{B})$ denote an automorphism induced by an order two automorphism of the graph associated to $B$ that swaps the two vertices. Then $\phi$ is order two, and it is not hard to check that for any $0$-ray $R(x,0)$, the image ray $\phi(R(x,0))$ is not equivalent to $R(x,0)$, and hence $\phi$ is not inert. (See \cite{FiebigPeriodic} for constraints on the actions of finite order automorphisms on periodic points.)\\

Given Corollary \ref{cor:entropyzero} and the above example, we pose the following question.

\begin{question}\label{question:zepowerinert}
If $(X_{A},\sigma_{A})$ is an irreducible shift of finite type of positive entropy, and $\phi \in \Aut(\sigma_{A})$ satisfies $h_{top}(\phi) = 0$, must $\phi^{k}$ be inert for some $k \ne 0$?
\end{question}

For the proof of Theorem \ref{thm:perronentropybound}, recall we have chosen a family of measures $\{\mu_{u}^{x}\}$ defined on the collection of unstable sets. We let $\tau_{\mu_{u}} \colon D_{A} \to \mathbb{R}$ denote the corresponding state as defined in \eqref{eqn:tracedefinition}.

\begin{proof}[Proof of Theorem \ref{thm:perronentropybound}]
Fix $n \in \mathbb{N}$. Choose a point $x \in X_{A}$, and consider the $0$-ray $R_{x} = R(x,0)$. There exists $k(n) \in \mathbb{N}$ for which $\phi^{n}(R_{x})$ is a $k(n)$-beam; for instance, Lemma \ref{lemma:rays} implies $k(n) \ge -W^{-}(n,\phi^{-1})$ works. Suppose $\phi^{n}(R_{x})$ is a union of $I(n)$ many $k(n)$-rays $V_{i}$
$$\phi^{n}(R_{x}) = \bigcup_{i=1}^{I(n)} V_{i}.$$
We may also write $R_{x}$ as a $k(n)$-beam, so there exists $J(n) \in \mathbb{N}$ such that $R_{x}$ is a union of $J(n)$ many $k(n)$-rays $W_{j}$
$$R_{x} = \bigcup_{j=1}^{J(n)} W_{j}.$$
Then we have
$$\tau_{\mu_{u}}(\phi^{n}(R_{x})) = \sum_{i=1}^{I(n)}\tau_{\mu_{u}}(V_{i}) = \sum_{i=1}^{I(n)}\lambda_{A}^{-k(n)}K_{i}$$
where, for each $i$, $K_{i}$ is an entry of the eigenvector $v_{r}$. On the other hand, since $\phi$ multiples the measure on unstable sets by $\lambda_{\phi}$, we also have
$$\tau_{\mu_{u}}(\phi^{n}(R_{x})) = \tau_{\mu_{u}}(\bigcup_{j=1}^{J(n)}\phi^{n}(W_{j})) = \sum_{j=1}^{J(n)}\tau_{\mu_{u}}(\phi^{n}(W_{j}))$$
$$=\sum_{j=1}^{J(n)} \lambda_{\phi}^{n}\mu_{u}(W_{j}) = \sum_{j=1}^{J(n)}\lambda_{\phi}^{n}\lambda_{A}^{-k(n)}K_{j}.$$

Thus
$$\sum_{j=1}^{J(n)}\lambda_{\phi}^{n}\lambda_{A}^{-k(n)}K_{j}= \sum_{i=1}^{I(n)}\lambda_{A}^{-k(n)}K_{i}$$
and it follows that there exists $K>0$ independent of $n$ such that
\begin{equation}\label{eqn:entropyineq1}
K \lambda_{\phi}^{n} J(n) \le I(n).
\end{equation}
Let $r$ denote the coding range of $\phi$. For a point $y \in X_{A}$, we can consider the collection of words
$$C^{\phi}_{y,n} = \{w_{i} \mid w_{i} = \phi^{i}(y)_{[k(n),k(n)+2r+1]}, \hspace{.03in} 0 \le i \le n\}.$$
Let $C^{\phi}(n)$ denote the (finite) set of all such collections, ranging over all points $y \in X_{A}$, so
$$C^{\phi}(n) = \{C^{\phi}_{y,n} \mid y \in X_{A} \}.$$
We claim a $k(n)$-ray $V_{i}$ is determined by a choice of $k(n)$-ray $W_{i}$ and a choice of $D \in C^{\phi}(n)$. Since $r$ is the coding range of $\phi$, given a $k(n)$-ray $W_{i}$ and $z \in W_{i}$, the ray $W_{i}$ together with the collection $C^{\phi}_{z,n} \in C^{\phi}(n)$ codes $\phi^{n}(z)|_{(-\infty,k(n)]}$. From this it follows that
$$I(n) \le J(n) \cdot \textnormal{card}(C^{\phi}(n)).$$
This inequality combined with \eqref{eqn:entropyineq1} gives
$$K \lambda_{\phi}^{n} J(n) \le J(n) \cdot \textnormal{card}(C^{\phi}(n))$$
and hence
$$\lambda_{\phi} \le \lim_{n \to \infty} \frac{1}{n} \log \textnormal{card}(C^{\phi}(n)).$$
Finally, note that, in general, we have
$$\lim_{n \to \infty}\frac{1}{n} \log \textnormal{card}(C^{\phi}(n)) \le h_{top}(\phi),$$
so
$$\lambda_{\phi} \le \lim_{n \to \infty} \frac{1}{n} \log \textnormal{card}(C^{\phi}(n)) \le h_{top}(\phi)$$
as desired.

\end{proof}
\subsection{Failure of an entropy conjecture}
In light of Theorem \ref{thm:perronentropybound}, a natural question is the following.\\

\textbf{Question EB (Entropy Bound) : } Let $(X_{A},\sigma_{A})$ be an irreducible shift of finite type, and suppose $\phi \in \Aut(\sigma_{A})$ is an automorphism. Does the inequality
\begin{equation}\label{eqn:entropyquestion}
\log \rho (S_{\phi}) \le h_{top}(\phi)
\end{equation}
hold?\\

We will show below that the answer to Question EB is, in general, \textbf{no}. However, before outlining examples where the inequality \eqref{eqn:entropyquestion} fails, let us take a moment to give some additional motivation for such a question.\\
\indent Question EB is in the spirit of Shub's classical entropy conjecture (\cite{Shub1974}, and \cite{KatokECRussian}). In the classical setting (e.g. Shub's conjecture), the core idea is whether the topological entropy of a diffeomorphism of a compact manifold is bounded below by the logarithm of the spectral radius of the induced action on the associated homology groups. One can ask whether there is some form of an analogous conjecture for automorphisms of a shift space.\\
\indent Of course, there are immediate difficulties in adapting such an entropy conjecture to the realm of shifts of finite type: topologically, $X_{A}$ is a Cantor set, and standard homology theories applied to $(X_{A},\sigma_{A})$ fail to provide meaningful information. However, the dimension group $\mathcal{G}_{A}$ can be considered as a certain homology group associated to the system $(X_{A},\sigma_{A})$, an idea which has been made rigorous, and greatly expanded upon, in Putnam's work on a homology theory for the more general category of Smale spaces \cite{PutnamHomology}. For a shift of finite type $(X_{A},\sigma_{A})$, Putnam's (unstable) homology groups vanish apart from the zeroth degree, where it agrees with $\mathcal{G}_{A}$. Thus for $(X_{A},\sigma_{A})$ there is only one non-zero Putnam homology group on which $\phi \in \Aut(\sigma_{A})$ acts, and this induced action is given precisely by the dimension representation applied to $\phi$. From this point of view, Question EB is rather natural.\\

\textbf{Examples for which EB fails:} We will construct a shift of finite type and automorphism for which \eqref{eqn:entropyquestion} does not hold. In fact, we will give a method with which one may produce many such examples. The idea is to find a primitive matrix over $\mathbb{Z}_{+}$ whose roots have a certain structure, described in detail below. To do this, we will make use of the affirmative answer, proved by Kim, Roush, and Ormes in \cite{KORSpectra}, to the Spectral Conjecture of Boyle and Handelman, in the case where the coefficient ring is $\mathbb{Z}$. We refer the reader to \cite{BHSpectra}, \cite{KORSpectra}, for details regarding this, along with other aspects of the primitive realization result used below. While it is possible (with some trial and error) to produce individual primitive matrices having the properties we want without using the results of \cite{KORSpectra}, we find the more general construction here to be worthwhile.\\
\indent To begin, let us suppose we have a polynomial $p(t) = \prod_{i=1}^{d}(t - \lambda_{i})$ in $\mathbb{Z}[t]$, whose roots $\{\lambda_{i}\}_{i=1}^{d}$ are non-zero and satisfy the following conditions\footnote{Conditions (1) and (2) are those which, according to the Spectral Conjecture of Boyle and Handelman (which is true over $\mathbb{Z}$), are sufficient for $\{\lambda_{i}\}$ to arise as the non-zero spectrum of a primitive matrix over $\mathbb{Z}$.}:
\begin{enumerate}
\item
$\lambda_{d} \in \mathbb{R}$ and $\lambda_{d} > \lvert \lambda_{i} \rvert$ for all $i \ne d$ (so $p(t)$ has a Perron root).
\item
For all $n \ge 1$, we have
$$\sum_{k | n}\left[\mu(\frac{n}{k})\sum_{i=1}^{d}\lambda_{i}^{k}\right] \ge 0,$$
where $\mu$ denotes the M\"obius function.
\item
There exists $k$ such that $\lvert \lambda_{k}^{-1} \rvert > \lambda_{d}$.
\end{enumerate}
For an explicit example of such a polynomial, let $p(t) = t^{3} - 5t^{2} - 6t + 1$.\\

By Theorem 2.2 in \cite{KORSpectra}, since conditions (1) and (2) above are satisfied by the roots of $p(t)$, there exists a primitive matrix $A$ over $\mathbb{Z}_{+}$ such that, for some $m \ge 0$, we have
$$\textnormal{det}(tI-A) = t^{m} \prod_{i=1}^{d}(t-\lambda_{i}).$$
Letting $(X_{A},\sigma_{A})$ denote the mixing shift of finite type associated to the matrix $A$, we have $h_{top}(\sigma_{A}) = \log \lambda_{j}$. Then the automorphism $\sigma_{A}^{-1} \in \Aut(\sigma_{A})$ satisfies
$$\log \rho(S_{\sigma_{A}^{-1}}) = \log \rho(\delta_{A}^{-1}) \ge \log \lvert \lambda_{k}^{-1} \rvert > \log \lambda_{j} = h_{top}(\sigma_{A}) = h_{top}(\sigma_{A}^{-1}).$$

\begin{remark}
It follows, by taking powers of a matrix $A$ produced in the Example above, that for any $R>0$ there exists an SFT $\sigma_{A}$ which contains an automorphism $\phi \in \Aut(\sigma_{A})$ for which
$$\log \rho(S_{\phi}) > h_{top}(\phi) + R.$$
\end{remark}

\section{Lyapunov Exponents}\label{sec:lyapunovbois}
We continue with the notation previously used, so $(X_{A},\sigma_{A})$ denotes an irreducible shift of finite type having positive entropy. Associated to an automorphism ${\phi \in \Aut(\sigma_{A})}$ are Lyapunov exponents $\alpha^{-}(\phi), \alpha^{+}(\phi)$, defined below, which, roughly speaking, measure rates of propagation of information of the automorphism $\phi$. These Lyapunov exponents were examined in the context of cellular automata in \cite{Shereshevsky1992} and \cite{Tisseur2000}, and recently in the more general setting of subshifts in \cite{CFKSpacetime} and \cite{NasuDegreesResolving}. Our treatment here follows more closely that of \cite{CFKSpacetime}. The quantities $\alpha^{-}(\phi), \alpha^{+}(\phi)$ also appear in the context of expansive subspaces as in \cite{BoyleLindExpansive}, and we refer the reader to \cite{CFKSpacetime} for more on this connection.\\
\indent The main result, Theorem \ref{thm:main} below, gives quantitative bounds relating the Lyapunov exponents of $\phi$ to the spectral radius of the action of $\phi$ on $(\mathcal{G}_{A},\mathcal{G}_{A}^{+})$. Though slightly technical in nature, the result places restrictions on the action of a range distorted automorphism on the associated dimension group.\\
\indent We begin by defining the quantities $\alpha^{-}, \alpha^{+}$.\\
\indent For $\phi \in \Aut (\sigma)$, we say $A \subset \mathbb{Z}$ \emph{$\phi$-codes} $B \subset \mathbb{Z}$ if whenever $x,y \in X$ satisfy $x_{i} = y_{i}$ for all $i \in A$, we have $(\phi(x))_{j} = (\phi(y))_{j}$ for all $j \in B$. Consider the sets
\begin{align*}
C^{-}(\phi)& = \{m \in \mathbb{Z} \mid (-\infty,0] \hspace{.07in} \phi \textnormal{-codes} \hspace{.07in} (-\infty,m]\} \\
C^{+}(\phi)& = \{m \in \mathbb{Z} \mid [0,\infty) \hspace{.07in} \phi \textnormal{-codes} \hspace{.07in} [m,\infty)\}.
\end{align*}
By the Curtis-Hedlund-Lyndon Theorem \cite{LindMarcus1995}, $\phi$ is given by a block code of some range, and it follows that both $C^{-}(\phi)$ and $C^{+}(\phi)$ are non-empty. We may then define the quantities
\begin{equation}
\begin{aligned}
W^{-}(n,\phi) = \textnormal{sup }C^{-}(\phi^{n})\\
W^{+}(n,\phi) = \textnormal{inf }C^{+}(\phi^{n}).
\end{aligned}
\end{equation}
In \cite{CFKSpacetime} it is shown that if $X_{A}$ is infinite, then $W^{-}(n,\phi) < \infty$ and $W^{+}(n,\phi) > -\infty$.\\
\indent Asymptotic information about the sequences $W^{-}(n,\phi), W^{+}(n,\phi)$ is captured by the following quantities.
\begin{definition}
Given $\phi \in \Aut(\sigma_{A})$, define
\begin{equation}\label{def:lyaps}
\begin{aligned}
\alpha^{-}(\phi) = \lim_{n \to \infty}\frac{W^{-}(n,\phi)}{n}\\
\alpha^{+}(\phi) = \lim_{n \to \infty}\frac{W^{+}(n,\phi)}{n}.
\end{aligned}
\end{equation}
\end{definition}
That $\alpha^{-}(\phi)$ and $\alpha^{+}(\phi)$ are both finite follows from the fact that both $\phi$ and $\phi^{-1}$ are each given by a block code of finite range.\\
\indent An automorphism $\phi \in \Aut(X,\sigma)$ is called \emph{range distorted} if
$$\alpha^{-}(\phi) = \alpha^{+}(\phi) = 0.$$
This notion of range distortion was introduced in \cite[Def. 5.8]{CFKSpacetime} (note our definition is not stated in the same language as theirs, but is equivalent; we refer the reader to their Proposition 5.12 to see this), but is related to older notions of automorphisms having a unique non-expansive subspace, as in \cite{BoyleLindExpansive}.\\

The following theorem gives a qualitative connection between the Lyapunov exponents $\alpha^{-}(\phi), \alpha^{+}(\phi)$ and the associated action of $\phi$ on the dimension group.
\begin{theorem}[\cite{CFKSpacetime}, Theorem 5.15]\label{thm:CyrKraFranks}
Let $(X_{A},\sigma_{A})$ be an irreducible shift of finite type such that $\textnormal{det}(I-tA)$ is an irreducible polynomial, and let $\phi \in \Aut(\sigma_{A})$. If both $\phi$ and $\phi^{-1}$ are range distorted, then $\phi$ is inert.
\end{theorem}

In this section we will prove the following, which generalizes Theorem \ref{thm:CyrKraFranks}.
\begin{theorem}\label{thm:unitcircle}
Let $(X_{A},\sigma_{A})$ be an irreducible shift of finite type, and let $\phi \in \Aut(\sigma_{A})$. If both $\phi$ and $\phi^{-1}$ are range distorted, then all eigenvalues of $\pi_{A}(\phi) = S_{\phi}$ lie on the unit circle.
\end{theorem}

\begin{remark}
To see how Theorem \ref{thm:CyrKraFranks} follows from Theorem \ref{thm:unitcircle}, suppose both $\phi$ and $\phi^{-1}$ are range distorted. Then Theorem \ref{thm:unitcircle} implies $\lambda_{\phi} = 1$. In the case $\textnormal{det}(I-tA)$ is irreducible, the map $\Psi$ from \eqref{eqn:autotracemap} is injective, and this implies $\phi$ is inert.
\end{remark}

It is not true in general that if $\phi \in \Aut(\sigma_{A})$ is an automorphism such that both $\phi$ and $\phi^{-1}$ are range distorted, then $\phi$ must be inert. In fact, there exist automorphisms of finite order which are not inert; an explicit example is given in Section \ref{sec:entropybound}. Such examples together with Theorem \ref{thm:CyrKraFranks} motivate the following question.\\

\begin{question}\label{question:rdpowerinert}
If $\phi \in \Aut(\sigma_{A})$ is an automorphism of an irreducible shift of finite type and both $\phi$ and $\phi^{-1}$ are range distorted, must $\phi^{k}$ be inert for some $k \ne 0$?
\end{question}

\begin{remark}\label{remark:connectionremark}
Question \ref{question:zepowerinert} concerns zero entropy automorphisms, while Question \ref{question:rdpowerinert} is concerned with range distorted automorphisms. Regarding the connection between zero entropy and range distortion, in \cite[Theorem 5.13]{CFKSpacetime} it is shown that if an automorphism $\phi$ is range distorted, then $h_{top}(\phi) = 0$. The converse of this is false; that is, there exist shifts of finite type having automorphisms $\phi$ for which $h_{top}(\phi) = 0$, but $\phi$ is not range distorted (we thank Ville Salo for pointing out an example of such an automorphism).
\end{remark}

We note that, since range distorted automorphisms have zero entropy (by \cite[Theorem 5.13]{CFKSpacetime}), a positive answer to Question \ref{question:zepowerinert} would imply a positive answer to Question \ref{question:rdpowerinert}.\\

Theorem \ref{thm:unitcircle} will follow from the more general Theorem \ref{thm:main} below, which gives bounds on the spectral radius of the action of an automorphism on the dimension group in terms of the associated Lyapunov exponents of the automorphism. We first introduce some more notation.\\
\indent Recall that for a $k \times k$ square matrix $A$ over $\mathbb{Z}_{+}$, we let $R(A)$ denote the eventual range subspace $\mathbb{Q}^{k}A^{k} \subset \mathbb{Q}^{k}$. Let $NS_{A} \colon R(A) \to R(A)$ denote the linear map induced by the action of $A$ on $R(A)$. Since $A$ is not nilpotent, $\textnormal{dim}R(A) \ge 1$ and the map $NS_{A}$ is invertible, and we let $\rho_{A}^{-}$ denote the spectral radius of the linear map ${NS_{A}^{-1} \otimes 1 \colon R(A) \otimes \mathbb{C} \to R(A) \otimes \mathbb{C}}$. We note that the linear maps $NS_{A}$ and $\delta_{A} \otimes 1 \colon \mathcal{G}_{A} \otimes \mathbb{Q} \to \mathcal{G}_{A} \otimes \mathbb{Q}$ are isomorphic, and $\rho_{A}^{-}$ may be computed directly from the matrix $A$: if $\lambda_{s}$ is an eigenvalue of the map $T_{A} \colon \mathbb{C}^{k} \to \mathbb{C}^{k}$ given by $x \mapsto Ax$, and $\lambda_{s}$ satisfies $0 < \lvert \lambda_{s} \rvert \le \lvert \lambda_{i} \rvert$ for all eigenvalues $\lambda_{i}$ of $T_{A}$, then $\rho_{A}^{-} = \lvert \lambda_{s} \rvert ^{-1}$.

\begin{theorem}\label{thm:main}
Let $(X_{A},\sigma_{A})$ be an irreducible shift of finite type, and let $\phi \in \Aut(\sigma_{A})$. Let $S_{\phi} = \pi_{A}(\phi)$ denote the image of $\phi$ under the dimension representation \\${\pi_{A} \colon \Aut(\sigma_{A}) \to \Aut(\mathcal{G}_{A},\mathcal{G}_{A}^{+},\delta_{A})}$, and let $\rho(S_{\phi})$ denote the spectral radius of the linear map $S_{\phi} \otimes 1 \colon \mathcal{G}_{A} \otimes \mathbb{C} \to \mathcal{G}_{A} \otimes \mathbb{C}$. Let $\rho_{A}^{-}$ denote the spectral radius of $NS_{A}^{-1}$. Then the following hold:
\begin{equation}\label{eqn:bound-}
\log \rho(S_{\phi}) \le \left[\lvert \alpha^{-}(\phi^{-1}) \rvert - \alpha^{-}(\phi)\right]h_{hop}(\sigma_{A}) + \lvert \alpha^{-}(\phi^{-1}) \rvert \log \rho_{A}^{-}.
\end{equation}
\begin{equation}\label{eqn:bound+}
\log \rho(S_{\phi}) \le \left[\lvert \alpha^{+}(\phi^{-1}) \rvert + \alpha^{+}(\phi)\right]h_{hop}(\sigma_{A}) + \lvert \alpha^{+}(\phi^{-1}) \rvert \log \rho_{A}^{-}.
\end{equation}
Moreover, both of the following hold:
\begin{enumerate}
\item
If $\alpha^{-}(\phi^{-1}) > 0$, then $\alpha^{-}(\phi) < 0$, and
\begin{equation}\label{eqn:earthbound-}
\log\rho(S_{\phi}) \le -\alpha^{-}(\phi)h_{top}(\sigma_{A}).
\end{equation}
\item
If $\alpha^{+}(\phi^{-1}) < 0$, then $\alpha^{+}(\phi) > 0$, and
\begin{equation}\label{eqn:earthbound+}
\log\rho(S_{\phi}) \le \alpha^{+}(\phi)h_{top}(\sigma_{A}).
\end{equation}
\end{enumerate}
\end{theorem}
\begin{remark}
There are cases where the bounds \eqref{eqn:bound-}, \eqref{eqn:bound+} become sharp: just consider the shift itself as an automorphism $\sigma_{n} \in \Aut(\sigma_{n})$, where $\sigma_{n}$ is a full shift. Furthermore, there are automorphisms for which the $\lvert \alpha^{\pm}(\phi^{-1}) \rvert$ terms in \eqref{eqn:bound-}, \eqref{eqn:bound+} are necessary. For example, let $(X_{A},\sigma_{A})$ denote the golden mean shift associated to the matrix $A = \begin{pmatrix} 1 & 1\\ 1 & 0 \end{pmatrix}$ whose spectral radius is $\lambda_{A} = \frac{1+\sqrt{5}}{2}$. For the automorphism $\tau = 1 \times \sigma_{A}^{-1}$ of the product system $(X_{A} \times X_{A},\sigma_{A} \times \sigma_{A})$, it is easy to check that $\alpha^{-}(\tau) = 0, \alpha^{-}(\tau^{-1}) = -1$, while ${\log \rho(S_{\tau}) = \log \lambda_{A}}$.
\end{remark}

In the case of a full shift, the bounds \eqref{eqn:bound-}, \eqref{eqn:bound+} in Theorem \ref{thm:main} become simpler.
\begin{corollary}\label{cor:waterhorse}
If $\phi \in \Aut(\sigma_{n})$ is an automorphism of the full shift on $n$ symbols, then both of the following hold:
\begin{equation}\label{eqn:skybound-}
\log\rho(S_{\phi}) \le -\alpha^{-}(\phi)h_{top}(\sigma_{A})
\end{equation}
\begin{equation}\label{eqn:skybound+}
\log\rho(S_{\phi}) \le \alpha^{+}(\phi)h_{top}(\sigma_{A}).
\end{equation}
Thus if $\phi$ satisfies either
$$\alpha^{-}(\phi) > 0$$
or
$$\alpha^{+}(\phi) < 0.$$
then $\phi$ is not inert.
\end{corollary}
\begin{proof}
If $(X_{n},\sigma_{n})$ is the full shift on $n$ symbols, we have
$$\log(\rho_{n}^{-})+h_{top}(\sigma_{n}) = \log(\frac{1}{n})+\log(n) = 0.$$
The result then follows from \eqref{eqn:bound-} and \eqref{eqn:bound+}.
\end{proof}
We observe that this cancellation $\log(\rho_{A}^{-}) + h_{top}(\sigma_{A}) = 0$ only occurs in the case $A$ is shift equivalent to a $1 \times 1$ matrix $(n)$ for some $n \in \mathbb{N}$.\\

We first prove Theorem \ref{thm:unitcircle} using Theorem \ref{thm:main}.
\begin{proof}[Proof of Theorem \ref{thm:unitcircle}]
If the automorphisms $\phi$ and $\phi^{-1}$ are both range distorted, then we have $\alpha^{-1}(\phi) = \alpha^{-}(\phi^{-1}) = 0$. Theorem \ref{thm:main} applied to both $\phi$ and $\phi^{-1}$ then implies $\rho(S_{\phi}) \le 1$ and $\rho(S_{\phi}^{-1}) \le 1$. This only happens if every eigenvalue of the linear map $S_{\phi} \otimes 1 \colon \mathcal{G}_{A} \otimes \mathbb{C} \to \mathcal{G}_{A} \otimes \mathbb{C}$ lies on the unit circle.
\end{proof}

The remainder of the section is devoted to the proof of Theorem \ref{thm:main}. Throughout, we will use $\mathcal{P}_{X_{A}}(n)$ to denote the number of admissible words of length $n$ in $(X_{A},\sigma_{A})$; for notational reasons, we define $\mathcal{P}_{X_{A}}(0)=1$. For a vector $v = (v_{1},\ldots,v_{k})$ we let $\lVert v \rVert_{1}$ denote the 1-norm, so $\lVert v \rVert_{1} = \sum_{1 \le i \le k} \lvert v_{i} \rvert$. For a linear map $T$ we let $\lVert T \rVert_{1}$ denote the operator norm induced by the 1-norm, i.e. $\lVert T \rVert_{1} = \textnormal{max}_{1 \le j \le k}\sum_{i=1}^{k}\lvert a_{ij} \rvert$.\\

First let us observe that for any $n \ge 1$ we have
\begin{equation}\label{eqn:smolthing-}
W^{-}(n,\phi) + W^{-}(n,\phi^{-1}) \le 0,
\end{equation}
\begin{equation}\label{eqn:smolthing+}
W^{+}(n,\phi) + W^{+}(n,\phi^{-1}) \ge 0,
\end{equation}
and hence
\begin{equation}\label{eqn:smollimitthing-}
\alpha^{-}(\phi) + \alpha^{-}(\phi^{-1}) \le 0,
\end{equation}
\begin{equation}\label{eqn:smollimitthing+}
\alpha^{+}(\phi) + \alpha^{+}(\phi^{-1}) \ge 0.
\end{equation}
One can derive these directly; a proof may be found in \cite[Prop. 3.13]{CFKSpacetime}, or \cite[Prop. 6.7]{NasuDegreesResolving}. The inequalities \eqref{eqn:smollimitthing-}, \eqref{eqn:smollimitthing+} give the first claims in parts (1) and (2) of Theorem \ref{thm:main}. \\

We now show how the inequalities \eqref{eqn:bound+}, \eqref{eqn:earthbound+} for $\alpha^{+}(\phi)$ follow from the inequalities \eqref{eqn:bound-}, \eqref{eqn:earthbound-} for $\alpha^{-}(\phi)$. Given $(X_{A},\sigma_{A})$, the \emph{reverse map}
$$r \colon (X_{A},\sigma_{A}^{-1}) \to (X_{A^{T}},\sigma_{A^{T}}), \hspace{.17in} r(x)_{i} = x_{-i},$$
is a topological conjugacy. Here by $A^{T}$ we mean the transpose of the matrix $A$. Since $\Aut(\sigma_{A}) = \Aut(\sigma_{A}^{-1})$ in a natural way, the reverse map $r$ gives an isomorphism
$$r^{*} \colon \Aut(\sigma_{A}) \to \Aut(\sigma_{A^{T}}), \hspace{.17in} r^{*}(\phi) = r\phi r^{-1}.$$
A quick check shows that for all $n \in \mathbb{N}$,
$$W^{-}(n,r^{*}(\phi)) = -W^{+}(n,\phi),$$
and hence
$$\alpha^{-}(r^{*}(\phi)) = -\alpha^{+}(\phi).$$
Then given $\phi \in \Aut(\sigma_{A})$, the inequalities \eqref{eqn:bound+}, \eqref{eqn:earthbound+} follow from applying \eqref{eqn:bound-}, \eqref{eqn:earthbound-} to the automorphism $r^{*}(\phi) \in \Aut(\sigma_{A^{T}})$, together with the observation that $h(\sigma_{A^{T}}) = h(\sigma_{A})$ and $\rho_{A}^{-} = \rho_{A^{T}}^{-}$.

\begin{remark}
Throughout the paper, we have only used the dimension group built from unstable sets. An analogous dimension group $\mathcal{G}_{A}^{s}$ built from stable sets may be similarly defined, leading to an alternative dimension representation ${\pi_{A}^{s} \colon \Aut(\sigma_{A}) \to \Aut(\mathcal{G}_{A}^{s})}$ of $\Aut(\sigma_{A})$ to the automorphisms of this stable dimension group. Since the Lyapunov exponent $\alpha^{-}(\phi)$ is defined in terms of the action of $\phi$ on unstable sets, and $\alpha^{+}(\phi)$ in terms of the action on stable sets, the quantity $\alpha^{-}(\phi)$ relates more immediately to the action of $\phi$ on the unstable dimension group, and $\alpha^{+}(\phi)$ to the action on the stable dimension group. We could have alternatively obtained the inequalities \eqref{eqn:bound+}, \eqref{eqn:earthbound+} for $\alpha^{+}$ using the stable dimension representation $\pi_{A}^{s}$. We feel it is worth giving a brief explanation why the reverse map $r$ used above accomplishes the same thing.\\
\indent Let us denote by $\pi_{A}^{u} \colon \Aut(\sigma_{A}) \to \Aut(\mathcal{G}_{A})$ the unstable dimension representation (i.e. the one used throughout the paper), and for $\phi \in \Aut(\sigma_{A})$ consider the corresponding complex linear maps
$$T_{\phi,\mathbb{C}}^{u} = \pi_{A}^{u}(\phi) \otimes 1 \colon \mathcal{G}_{A} \otimes \mathbb{C} \to \mathcal{G}_{A} \otimes \mathbb{C},$$
$$T_{\phi,\mathbb{C}}^{s} = \pi_{A}^{s}(\phi) \otimes 1 \colon \mathcal{G}_{A}^{s} \otimes \mathbb{C} \to \mathcal{G}_{A}^{s} \otimes \mathbb{C}.$$
If $T \colon V \to V$ is a linear map of complex vector spaces we let $T^{*}$ denote the dual map, i.e. the map $T^{*} \colon \textnormal{Hom}(V,\mathbb{C}) \to \textnormal{Hom}(V,\mathbb{C})$ induced by $T$. Then the relationship between the two dimension representations may be summarized as follows: for $\phi \in \Aut(\sigma_{A})$ we have
\begin{equation}\label{eqn:unstablestablefrends}
T_{\phi,\mathbb{C}}^{s} = (T_{\phi,\mathbb{C}}^{u})^{*}.
\end{equation}
Since our results concern the spectra of the maps $T_{\phi,\mathbb{C}}^{u}$ and $T_{\phi,\mathbb{C}}^{s}$, and these spectra are the same (by \eqref{eqn:unstablestablefrends}), we can derive results regarding $\alpha^{+}(\phi)$ using $\alpha^{-}(\phi)$ and $\pi_{A}^{u}(\phi)$, together with this duality. For more on this duality, we refer the reader to \cite{PutnamKilloughModules}.
\end{remark}

Continuing with the proof of Theorem \ref{thm:main}, the following lemma will play an important role.

\begin{lemma}\label{lemma:rays}
Let $R(x,0)$ be any $0$-ray in $X_{A}$. For any $n \ge 1$, $\phi^{n}(R(x,0))$ is a $-W^{-}(n,\phi^{-1})$-beam. Moreover, the following holds:
\begin{enumerate}
\item
If $W^{-}(n,\phi^{-1}) \le 0$, then $\phi^{n}(R(x,0))$ is a union of at most $$\mathcal{P}_{X_{A}}(\lvert W^{-}(n,\phi^{-1}) \rvert -W^{-}(n,\phi))$$
many distinct $-W^{-}(n,\phi^{-1})$-rays.
\item
If $W^{-}(n,\phi^{-1}) \ge 0$, then $\phi^{n}(R(x,0))$ is a union of at most $$\mathcal{P}_{X_{A}}(\lvert W^{-}(n,\phi) \rvert -W^{-}(n,\phi^{-1}))$$ many distinct $-W^{-}(n,\phi^{-1})$-rays.
\end{enumerate}
\end{lemma}
\begin{proof}
Suppose first we are in case (1), so $W^{-}(n,\phi^{-1}) \le 0$. Given $y \in \phi^{n}(R(x,0))$, we claim
$$\{z \in X_{A} \mid z_{(-\infty,\lvert W^{-}(n,\phi^{-1}) \rvert]} = y_{(-\infty, \lvert W^{-}(n,\phi^{-1}) \rvert]} \} \subset \phi^{n}(R(x,0)).$$
Indeed, for such a $z$, we have $\phi^{-n}(z)$ and $\phi^{-n}(y)$ agree on $(-\infty,0]$. But $\phi^{-n}(y)$ and $x$ agree on $(-\infty,0]$, so $\phi^{-n}(z)$ and $x$ agree on $(-\infty,0]$. Then $\phi^{-n}(z) \in R(x,0)$, and $z \in \phi^{n}(R(x,0))$.\\
\indent The set of words $\mathcal{W}_{n} = \{w = z_{[W^{-}(n,\phi),\lvert W^{-}(n,\phi^{-1}) \rvert]} \mid z \in \phi^{n}(R(x,0))\}$ is finite and non-empty, and we define, for $w \in \mathcal{W}_{n}$, the set
$$B_{w} = \{z \mid z_{(-\infty,W^{-}(n,\phi)-1]} = \phi^{n}(x)_{(-\infty,W^{-}(n,\phi)-1]} \textnormal{ and } z_{[W^{-}(n,\phi),\lvert W^{-}(n,\phi^{-1}) \rvert]} = w \}.$$
Then each $B_{w}$ is a $\lvert W^{-}(n,\phi^{-1}) \rvert$-ray, and we have
$$\phi^{n}(R(x,0)) = \bigcup_{w \in \mathcal{W}_{n}}B_{w}.$$
Case (1) then follows since $\lvert \mathcal{W}_{n} \rvert \le \mathcal{P}_{X_{A}}(\lvert W^{-}(n,\phi^{-1}) \rvert -W^{-}(n,\phi))$.\\
\indent The proof of case (2) is analogous to that of case (1); just replace every occurrence of $\lvert W^{-}(n,\phi^{-1}) \rvert$ in the proof above with $-W^{-}(n,\phi^{-1})$.
\end{proof}
We find the following notation to be convenient, and will use it throughout the remainder of the proof.\\

\indent \textbf{Notation: } Given $\phi \in \Aut(\sigma_{A})$, we define sequences
$$A^{-}(n) = \lvert W^{-}(n,\phi^{-1}) \rvert -W^{-}(n,\phi),$$
$$A^{+}(n) = \lvert W^{-}(n,\phi) \rvert  - W^{-}(n,\phi^{-1})$$
and note that by \eqref{eqn:smolthing-} and \eqref{eqn:smolthing+} these sequences are both non-negative.
\begin{proof}[Proof of Theorem \ref{thm:main}]
For each $1 \le i \le k$, fix some $x^{(i)} \in X_{A}$ such that the edge corresponding to $x^{(i)}_{0}$ ends at state $i$. Let $U_{i}$ denote the 0-beam which consists of the single 0-ray $R(x^{(i)},0)$. Thus $v_{U_{i},0}$ is the $i$th standard basis (row) vector $e_{i}$ in $\mathbb{Z}^{k}$. Using $S_{\phi}$ we define a linear map
\begin{align*}
T_{\phi} \colon \mathbb{C}^{k} \to \mathbb{C}^{k}
\end{align*}
by
$$e_{i} \mapsto S_{\phi}(\delta_{A}^{-k}e_{i}A^{k})$$
(and then extending using linearity). By construction, $\mathcal{G}_{A} \otimes \mathbb{C}$ may be identified with a $T_{\phi}$-invariant subspace of $\mathbb{C}^{k}$ on which the action of $T_{\phi}$ is isomorphic to the map $S_{\phi} \otimes 1 \colon \mathcal{G}_{A} \otimes \mathbb{C} \to \mathcal{G}_{A} \otimes \mathbb{C}$.\\
\indent We claim there exists a constant $K > 0$ such that, for any $1 \le i \le k$,
\begin{equation}\label{eqn:tempehbound}
\lVert T^{n}_{\phi}e_{i} \rVert _{1} \le K \lVert \delta_{A}^{W^{-}(n,\phi^{-1})} \rVert _{1}\lVert v_{\phi^{n}(U_{i}),-W^{-}(n,\phi^{-1})}\rVert _{1}.
\end{equation}
To see this, observe that by \eqref{diagram}, for any $1 \le i \le k$ we have
$$T_{\phi}^{n}e_{i} = S_{\phi}^{n}(\delta_{A}^{-k}e_{i}A^{K})= S_{\phi}^{n} \theta([U_{i}]) = \theta([\phi^{n}(U_{i})]).$$
By Lemma \ref{lemma:rays}, $\phi^{n}(U_{i})$ is a $-W^{-}(n,\phi^{-1})$-beam, and hence
$$\lVert T^{n}_{\phi}e_{i} \rVert _{1}= \lVert \theta([\phi^{n}(U_{i})])\rVert _{1} = \lVert \delta_{A}^{-k+W^{-}(n,\phi^{-1})}v_{\phi^{n}(U_{i}),-W^{-}(n,\phi^{-1})}A^{k}\rVert _{1}$$
$$\le K \lVert \delta_{A}^{W^{-}(n,\phi^{-1})}v_{\phi^{n}(U_{i}),-W^{-}(n,\phi^{-1})}\rVert _{1}  \le K \lVert \delta_{A}^{W^{-}(n,\phi^{-1})}\rVert _{1}\lVert v_{\phi^{n}(U_{i}),-W^{-}(n,\phi^{-1})}\rVert _{1},$$
giving \eqref{eqn:tempehbound}. Note that by Lemma \ref{lemma:rays}, given $n \in \mathbb{N}$, for any $1 \le i \le k$ we have the following:
\begin{equation}\label{eqn:onesnek}
\textnormal{ if } W^{-}(n,\phi^{-1}) \le 0, \textnormal{ then } \lVert v_{\phi^{n}(U_{i}),-W^{-}(n,\phi^{-1})} \rVert _{1} \le \mathcal{P}_{X_{A}}(A^{-}(n)).
\end{equation}
\begin{equation}\label{eqn:twosnek}
\textnormal{ if } W^{-}(n,\phi^{-1}) \ge 0, \textnormal{ then } \lVert v_{\phi^{n}(U_{i}),-W^{-}(n,\phi^{-1})} \rVert_{1} \le \mathcal{P}_{X_{A}}(A^{+}(n)).
\end{equation}
When $W^{-}(n,\phi^{-1}) \ge 0$, it follows from \eqref{eqn:smolthing-} that we must have $A^{+}(n) \le A^{-}(n)$, which gives $\mathcal{P}_{X_{A}}(A^{+}(n)) \le \mathcal{P}_{X_{A}}(A^{-}(n))$ in this case. From this, together with \eqref{eqn:tempehbound}, \eqref{eqn:onesnek} and \eqref{eqn:twosnek} we get
\begin{equation}\label{eqn:greensnek}
\textnormal{ if } W^{-}(n,\phi^{-1}) \ge 0, \textnormal{ then } \lVert T^{n}_{\phi} \rVert _{1} \le K \lVert \delta_{A}^{W^{-}(n,\phi^{-1})} \rVert_{1} \mathcal{P}_{X_{A}}(A^{+}(n)).
\end{equation}
\begin{equation}\label{eqn:bluesnek}
\textnormal{ for any } n \in \mathbb{N}, \hspace{.11in}  \lVert T^{n}_{\phi} \rVert_{1} \le K \lVert \delta_{A}^{W^{-}(n,\phi^{-1})}\rVert _{1} \mathcal{P}_{X_{A}}(A^{-}(n)).
\end{equation}
We now proceed by cases.\\
\indent \emph{Case 1: $\alpha^{-}(\phi^{-1}) > 0$.}\\
\indent In this case, for sufficiently large $n$ we have $W^{-}(n,\phi^{-1}) > 0$, and it follows from \eqref{eqn:bluesnek} that
$$\lVert T^{n}_{\phi} \rVert_{1} \le K \lVert \delta_{A}^{W^{-}(n,\phi^{-1})}\rVert _{1} \mathcal{P}_{X_{A}}(A^{+}(n)),$$
and
$$ \frac{1}{n} \log \lVert T_{\phi}^{n} \rVert _{1} \le \frac{1}{n} \log K + \frac{1}{n} \log \lVert \delta_{A}^{W^{-}(n,\phi^{-1})}\rVert _{1} + \frac{1}{n} \log  \mathcal{P}_{X_{A}}(A^{+}(n)).$$
By Gelfand's Formula \cite[Section 17.1]{LaxBook} we have
$$\lim_{n \to \infty} \frac{1}{n}\log \lVert T_{\phi}^{n} \rVert _{1} = \log \rho(T_{\phi}),$$
so it suffices to consider the two non-trivial terms on the right hand side,
$$\lim_{n \to \infty} \frac{1}{n} \log \lVert \delta_{A}^{W^{-}(n,\phi^{-1})}\rVert _{1}, \hspace{.3in} \lim_{n \to \infty} \frac{1}{n} \log  \mathcal{P}_{X_{A}}(A^{+}(n)).$$
For the first term, we observe that
\begin{equation}\label{eqn:thesmalldeer}
\lim_{n \to \infty} \frac{1}{n} \log \lVert \delta_{A}^{W^{-}(n,\phi^{-1})}\rVert _{1} = \alpha^{-}(\phi^{-1}) \log \rho(\delta_{A}) = \alpha^{-}(\phi^{-1})h_{top}(\sigma_{A})
\end{equation}

For the second term, we first recall that (see \cite[Sec. 6.3]{LindMarcus1995})
\begin{equation}\label{eqn:entswalking}
\lim_{n \to \infty}\frac{\textnormal{log}\mathcal{P}_{X_{A}}(n)}{n} = h_{top}(\sigma_{A}).
\end{equation}
\indent We wish to show that
\begin{equation}\label{eqn:somebound}
\lim_{n \to \infty} \frac{1}{n} \textnormal{log} \mathcal{P}_{X_{A}}(A^{+}(n)) = \left[ \lvert \alpha^{-}(\phi) \rvert -\alpha^{-}(\phi^{-1})\right]h_{top}(\sigma_{A}).
\end{equation}
If $\lvert \alpha^{-}(\phi) \rvert = \alpha^{-}(\phi^{-1})$ then this holds, since, noting that
$$\lim_{n \to \infty}\frac{1}{n}A^{+}(n) = \lvert \alpha^{-}(\phi) \rvert -\alpha^{-}(\phi^{-1})$$
the left hand side of \eqref{eqn:somebound} is then also zero. On the other hand, if $\lvert \alpha^{-}(\phi) \rvert \ne \alpha^{-}(\phi^{-1})$ then we must have $\lvert \alpha^{-}(\phi) \rvert > \alpha^{-}(\phi^{-1})$ (by \eqref{eqn:smolthing-}), and hence $A^{+}(n) \to \infty$, so

$$\lim_{n \to \infty}\frac{1}{n}\textnormal{log}\mathcal{P}_{X_{A}}(A^{+}(n))$$
$$= \lim_{n \to \infty} \left( \frac{A^{+}(n)}{n} \right) \left( \frac{\textnormal{log} \mathcal{P}_{X_{A}}(A^{+}(n))}{A^{+}(n)} \right)$$
$$=\left[\lvert \alpha^{-}(\phi) \rvert -\alpha^{-}(\phi^{-1})\right]h_{top}(\sigma_{A}).$$
Putting \eqref{eqn:somebound} and \eqref{eqn:thesmalldeer} together gives
$$\log \rho(T_{\phi}) \le \lvert \alpha^{-}(\phi) \rvert h_{top}(\phi)$$
which completes Case 1.\\

\emph{Case 2: $\alpha^{-}(\phi^{-1}) \le 0.$}\\
First we note that in this case, we must have $W^{-}(n,\phi^{-1}) \le 0 $ for all $n$. Indeed, the sequence $W^{-}(n,\phi^{-1})$ is super-additive (see \cite[Lemma 3.10]{CFKSpacetime}), and by Fekete's Lemma $\alpha^{-}(\phi^{-1}) = \underset{n}\sup \frac{W^{-}(n,\phi^{-1})}{n} $. The case then proceeds similarly to the previous. From \eqref{eqn:bluesnek} we have that
$$\lVert T^{n}_{\phi} \rVert _{1} \le K \lVert \delta_{A}^{W^{-}(n,\phi^{-1})} \rVert _{1} \mathcal{P}_{X_{A}}(A^{-}(n)),$$
and
$$ \frac{1}{n} \log \lVert T_{\phi}^{n} \rVert _{1} \le \frac{1}{n} \log K + \frac{1}{n} \log \lVert \delta_{A}^{W^{-}(n,\phi^{-1})} \rVert _{1} + \frac{1}{n} \log  \mathcal{P}_{X_{A}}(A^{-}(n)).$$
We again consider the two non-trivial terms on the right hand side,
$$\lim_{n \to \infty} \frac{1}{n} \log \lVert \delta_{A}^{W^{-}(n,\phi^{-1})} \rVert _{1}, \hspace{.3in} \lim_{n \to \infty} \frac{1}{n} \log  \mathcal{P}_{X_{A}}(A^{-}(n)).$$
For the first term, we claim that
\begin{equation}\label{eqn:snekfrends}
\lim_{n \to \infty} \frac{1}{n} \log \lVert \delta_{A}^{W^{-}(n,\phi^{-1})} \rVert _{1} = -\alpha^{-}(\phi^{-1})\log \rho(\delta_{A}^{-1}) = \lvert \alpha^{-}(\phi^{-1}) \rvert \log \rho(\delta_{A}^{-1}).
\end{equation}
If $\alpha^{-}(\phi^{-1}) < 0$, this is clear from Gelfand's formula. If $\alpha^{-}(\phi^{-1}) = 0$, it also follows from Gelfand's formula, since then
$$\lim_{n \to \infty} \frac{1}{n} \log \lVert \delta_{A}^{W^{-}(n,\phi^{-1})} \rVert = 0$$
and both sides are zero.

The second term is analogous to \eqref{eqn:somebound}, and we get
\begin{equation}\label{eqn:squirrelo}
\lim_{n \to \infty} \frac{1}{n} \log  \mathcal{P}_{X_{A}}(A^{-}(n)) = \left[\lvert \alpha^{-}(\phi^{-1}) \rvert - \alpha^{-}(\phi)\right]h_{hop}(\sigma_{A}).
\end{equation}
Putting together \eqref{eqn:snekfrends} and \eqref{eqn:squirrelo} gives
$$\log \rho(T_{\phi}) \le \left[\lvert \alpha^{-}(\phi^{-1})\rvert - \alpha^{-}(\phi)\right]h_{hop}(\sigma_{A}) + \lvert \alpha^{-}(\phi^{-1}) \rvert \log \rho(\delta_{A}^{-1}).$$
which completes Case 2.


\end{proof}
This concludes the proof of Theorem \ref{thm:main}.\\



\section{Distortion}
For a countable group $G$, we call an element $g \in G$ \emph{group distorted} if there exists some finite set $F \subset G$ such that
$$\lim_{n \to \infty} \frac{L_{F}(g^{n})}{n} = 0$$
where $L_{F}(g^{k})$ denotes the length (in the word metric) of the shortest presentation of $g^{k}$ in the subgroup $\langle F \rangle$ generated by $F$ in $G$. While any element of finite order is necessarily group distorted, there exists groups which contain group distorted elements of infinite order. For example, the groups $\textnormal{SL}(k,\mathbb{Z})$ for $k \ge 3$ contain group distorted elements of infinite order \cite{LMR2000}. It is not hard to show (see \cite[Prop 3.4]{CFKPDistortion} for example) that if $\phi \in \Aut(\sigma)$ is group distorted, then it is range distorted. The following question was asked in \cite{CFKPDistortion}:\\

\textbf{Question 5.3 in \cite{CFKPDistortion}: } Does a group having group distorted elements of infinite order embed into the automorphism group of some positive entropy subshift?\\

The dimension representation of group distorted elements in $\Aut(\sigma_{A})$ must have spectrum on the unit circle. Indeed, since group distorted elements in $\Aut(\sigma_{A})$ are necessarily range distorted (see \cite[Prop. 3.4]{CFKPDistortion}), this can be deduced from Theorem \ref{thm:unitcircle}. However, we give a direct proof below which is more elementary.
\begin{proposition}\label{prop:distortedinertness}
If $\phi \in \Aut(\sigma_{A})$ is group distorted with respect to a finite generating set $F \subset \Aut(\sigma_{A})$, then $\log \rho(\pi_{A}(\phi)) = 0$.
\end{proposition}
\begin{proof}
If $\phi \in \Aut(\sigma_{A})$ is group distorted, then its image $S_{\phi} = \pi_{A}(\phi)$ under the dimension representation must also be group distorted. Choosing generators $\{T_{i}\}_{i=1}^{k}$ for $\pi_{A}(F)$, for any $n \in \mathbb{N}$ there exists $l(n)$ for which $S_{\phi}^{n}$ is a product of $l(n)$ of matrices of the form $T_{i}^{\epsilon}$, where $\epsilon = \pm 1$, and $\frac{l(n)}{n} \to 0$. If $M = \max\{\lVert T_{i} \rVert , \lVert T_{i} \rVert ^{-1}\}_{i=1}^{k}$, then
$$\lVert S_{\phi}^{n} \rVert \le M^{l(n)},$$
and hence
$$\log \rho (S_{\phi}) = \lim_{n \to \infty} \frac{1}{n} \log \lVert S_{\phi}^{n} \rVert \le \lim_{n \to \infty} \frac{l(n)}{n} \log M = 0.$$
Applying the above to $\phi^{-1}$ as well gives $\log \rho (S_{\phi}) = 0$.
\end{proof}
While group distorted elements in $\Aut(\sigma_{A})$ are necessarily range distorted, to the author's knowledge, it is not known whether $\phi \in \Aut(\sigma_{A})$ being range distorted implies $\phi$ must be group distorted. A consequence of Theorem \ref{thm:unitcircle} is that if $\phi$ is range distorted then we must still have $\log \rho(\pi_{A}(\phi)) = 0$.\\

\bibliographystyle{plain}
\bibliography{ssbib}

\end{document}